\documentclass{amsart}
\usepackage[english]{babel}
\usepackage{pgfplots}
\pgfplotsset{compat=1.15}
\usepackage{graphicx,float,wrapfig, caption}
\usepackage[percent]{overpic}
\usetikzlibrary{positioning}
\usepackage{standalone} 
\usepackage{subcaption} 
\usepackage{mathrsfs}
\usetikzlibrary{arrows}
\usepackage{mathrsfs, mathtools, amssymb}
\usepackage{amsthm}
\usepackage{dsfont}
\graphicspath{ {pictures/}{tikz/} }

\usepackage{tikz, tikz-cd}

\usepackage[paper=a4paper, margin=2cm]{geometry}

\usepackage{hyperref}
\hypersetup{
    unicode=false,          % non-Latin characters in bookmarks
    pdftoolbar=true,        % show toolbar?
    pdfmenubar=true,        % show menu?
    pdffitwindow=false,     % window fit to page when opened
    pdfstartview={FitH},    % fits the width of the page to the window
    pdftitle={Faces of cosmological polytopes},    % title
    pdfauthor={Lukas K\"uhne and Leonid Monin},     % author
    pdfkeywords={Cosmological polytope}, % list of keywords
    pdfnewwindow=true,      % links in new window
    colorlinks=true,       % false: boxed links; true: colored links
    linkcolor=blue,          % color of internal links
    citecolor=red,        % color of links to bibliography
    filecolor=magenta,      % color of file links
    urlcolor=cyan           % color of external links
}
\usepackage{thmtools}

\newenvironment{nouppercase}{%
	\renewcommand{\uppercasenonmath}[1]{}}{}
\theoremstyle{plain}
\newtheorem{theorem}{Theorem}[section]

\newtheorem{proposition}[theorem]{Proposition}
\newtheorem{corollary}[theorem]{Corollary}

\theoremstyle{definition}
\newtheorem{definition}[theorem]{Definition}
\newtheorem{remark}[theorem]{Remark}
\newtheorem{example}[theorem]{Example}

\newcommand{\rleft}{\mathopen{}\mathclose\bgroup\left}
\newcommand{\rright}{\aftergroup\egroup\right}

\newcommand{\R}{{\mathbb{R}}}
\newcommand{\Z}{{\mathbb{Z}}}

\DeclareMathOperator{\conv}{conv}

\newcommand{\Vol}{{\mathrm{Vol}}}

\begin{document}
\title{Faces of Cosmological Polytopes}
\author[Lukas K\"uhne]{Lukas K\"uhne}
\address{Fakult\"at f\"ur Mathematik, Universit\"at Bielefeld, Bielefeld, Germany}
\email{lukas.kuehne@math.uni-bielefeld.de}
\author{Leonid Monin}
\address{Institute of Mathematics, EPFL, Lausanne, Switzerland}
\email{leonid.monin@epfl.ch}
\date{}
	\begin{nouppercase}
		\maketitle
	\end{nouppercase}

\begin{abstract}
A cosmological polytope is a lattice polytope introduced by Arkani-Hamed, Benincasa, and Postnikov in their study of the wavefunction of the universe in a class of cosmological models.
More concretely, they construct a cosmological polytope for any Feynman diagram, i.e. an undirected graph.
In this paper, we initiate a combinatorial study of these polytopes.
We give a complete description of their faces, identify minimal faces that are not simplices and compute the number of faces in specific instances. In particular, we give a recursive description of the $f$-vector of cosmological polytopes of trees.
\end{abstract}

\section{Introduction}
Arkani-Hamed, Benincasa, and Postnikov defined a cosmological polytope $P_G$ for every undirected graph $G=(V,E)$, that is $V=\{v_1,\dots,v_k\}$ is a finite set of \emph{vertices} and $E=\{e_1,\dots,e_n\}$ a finite set of \emph{edges} with $e_i=\{v_{j_1},v_{j_2}\}$ for some $1\le j_1,j_2\le k$.
Throughout this article we work in the space $\R^{|V|+|E|}$ with standard basis vectors $\mathbf{x}_{v_i}$, $\mathbf{y}_{e_j}$  for $1\le i\le k$ and $1\le j\le n$. 
\begin{definition}[\cite{arkani2017cosmological}]\label{def:cosm}
The cosmological polytope $P_G$ associated with a graph $G=(V,E)$ is the convex hull of the following $3|E|+|V|$ vertices
\[
P_G = \conv\left( \bigcup_{e=\{v,w\}\in E} \{\mathbf{y}_{e}+\mathbf{x}_v-\mathbf{x}_w, \mathbf{y}_{e}-\mathbf{x}_v+\mathbf{x}_w, -\mathbf{y}_{e}+\mathbf{x}_v+\mathbf{x}_w\}\cup \bigcup_{v\in V}\mathbf{x}_v\right).
\]
For an edge $e=\{v,w\}$ we will denote the above points in $\R^{n+k}$ by 
\begin{align*}
p_e=-\mathbf{y}_{e}+\mathbf{x}_v+\mathbf{x}_w,\quad
p_{e,w}=\mathbf{y}_{e}-\mathbf{x}_v+\mathbf{x}_w,\quad 
p_{e,v}=\mathbf{y}_{e}+\mathbf{x}_v-\mathbf{x}_w.
\end{align*}
\end{definition}

\begin{remark}
Definition~\ref{def:cosm} is slightly different from the standard definition of cosmological polytopes in \cite{arkani2017cosmological}, where a cosmological polytope is defined as a convex hull of vertices $p_e, p_{e,v}, p_{e,w}$ only. The two definitions coincide in the case of  graphs with no isolated vertices since $\mathbf{x}_v\in \conv(p_e, p_{e,v}, p_{e,w})$ for any edge $e=\{v,w\}$ of~$G$. However, Definition~\ref{def:cosm} works better for  graphs with isolated vertices, which might appear in the recursion in Section~\ref{sec:trees}.
\end{remark}

\begin{example}
The cosmological polytope of a graph consisting of two parallel edges $e,e'$ between two vertices $v,w$ is a prism over a triangle which is depicted in Figure~\ref{fig:polytope}.
\end{example}

% \begin{figure}[H]
%  \begin{subfigure}{.2\textwidth}
%    \centering
%    \includestandalone[width=\linewidth]{G22}
%  \end{subfigure}%
%  \begin{subfigure}
%   \centering
%   \includestandalone[width=.3\linewidth]{prism}
%   \end{subfigure}%
% \caption{The cosmological polytope of a pair of two parallel edges.}
%   \label{fig:polytope}
% \end{figure}

\begin{figure}[H]
		\begin{subfigure}[b]{0.48\linewidth}
			\centering
            \includegraphics[scale=1]{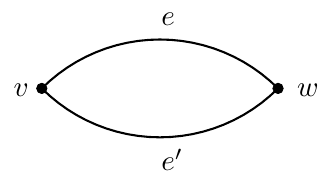}
			\label{fig:ex_general}
		\end{subfigure}
		% \begin{subfigure}[b]{0.2\linewidth}
		% 	\centering
		% 	%\includegraphics[scale=.45]{pics/ex1}
		% 	\label{fig:ex1}
		% \end{subfigure}
	\begin{subfigure}[b]{0.48\linewidth}
		\centering
        \includegraphics[scale=.8]{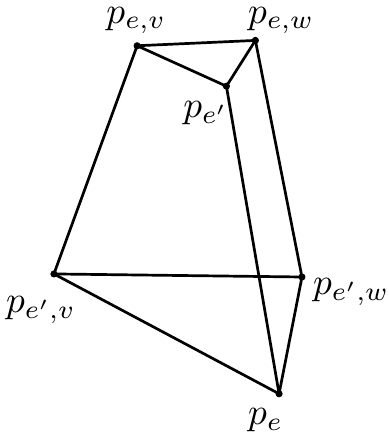}
		\label{fig:rex3}
	\end{subfigure}
		\caption{The graph on the left corresponds to the cosmological polytope on the right.}
		\label{fig:polytope}
	\end{figure}

\subsection{Physical perspective}
In recent years a connection between the physics of scattering amplitudes and a class of mathematical objects called \emph{positive geometries} was discovered \cite{arkani2017positive}.
Positive geometries can be thought of as a vast generalization of convex polytopes, they encompass objects such as polytopes, the positive Grassmannian, and tree and loop Amplituhedra \cite{arkani2014amplituhedron,bai2016amplituhedron}.
The connection of positive geometries to physics is usually via a top-dimensional form uniquely determined by the condition that it has logarithmic singularities (only) along all boundary components of a positive geometry. Thus computing the canonical form is the central goal in the studies of positive geometries.

There are two standard ways to compute the canonical form of a positive geometry. The first method is to find a subdivision of a positive geometry $X$ into simpler positive geometries $Y_1,\ldots Y_k$. In this case, the canonical form ${\bf\Omega}_X$ of $X$ is given as the sum
\[
{\bf\Omega}_X = {\bf\Omega}_{Y_1} + \ldots + {\bf\Omega}_{Y_k}.
\]
This strategy provides a direct way to get an expression for the canonical form and was successfully applied in a number of situations \cite{BT22,tr3,tr1,lukowski2020positive,tr2,tr4}.

The disadvantage of this method is that one does not obtain a closed formula for the canonical form. Thus it is sometimes more convenient to compute the canonical form directly from its definition. More concretely, let $D_1,\ldots,D_r$ be the boundary components of a positive geometry $X$ defined by polynomials $f_1,\ldots,f_r$ respectively. Then the condition on the singularities of the canonical form guarantees that it can be written as
\[
{\bf\Omega}_X =\frac{g}{f_1\cdot \ldots \cdot f_r} \omega,
\]
where $g$ is some polynomial and $\omega$ is a regular form on $X$. Thus the problem of computing ${\bf\Omega}_X$ boils down to the computation of the numerator polynomial $g$. Moreover, the polynomial $g$ is determined by the condition that it should cancel on the poles of $1/(f_1\cdot \ldots \cdot f_r)$ outside of $X$, i.e. $g$ should vanish along the intersections of the $D_i$'s outside of $X$ \cite{arkani2015positive}. This observation leads to explicit formulas for the canonical form. One particular example is \cite{kohn2021adjoints} where the numerator of the canonical form of a plane positive geometry was identified as the adjoint curve to the boundary.

In \cite{arkani2017cosmological}, it was noticed that the connection between physics and positive geometries extends further to cosmology. More concretely, the cosmological polytope is constructed in \cite{arkani2017cosmological} as the positive geometric counterpart to the physics of cosmological time evolution and the wavefunction of the universe. This motivates the study of subdivisions and the face structure of cosmological polytopes as their facets are the components $D_1,\dots,D_r$ that are relevant for the computation of $\mathbf{\Omega}_X$ as discussed above.

\subsection{Combinatorial perspective}

There are several constructions of polytopes arising from graphs.
The most relevant to cosmological polytopes are \emph{symmetric edge polytopes} which recently gained considerable attention~\cite{symmetric_edge3,symmetric_edge2,symmetric_edge1}.
In particular, the symmetric edge polytope is the image of a linear projection of a facet of the cosmological polytope (the scattering facet) of the same graph. Moreover, this projection sends the vertices of the scattering facet to the vertices of the symmetric edge polytope. Thus the information on the faces of cosmological polytopes can be used to study coherent subdivisions of symmetric edge polytopes.

\subsection{Our contribution}

In this paper, we start a comprehensive study of the faces of cosmological polytopes.
Concretely, we give a criterion for a subset of vertices of $P_G$ to form a face in Section~\ref{sec:gen_faces}.
This criterion can be checked easily by considering basic properties of the graph $G$.
As explained above, knowledge of the faces of~$P_G$ is relevant to determine both the numerator and denumerator of the canonical form of the polytope.

Subsequently, in Section~\ref{sec:special_faces}, we describe two special families of faces of $P_G$, corresponding to the vertices and cycles of $G$, respectively.
Our general face criterion yields that the faces of these two families are exactly the minimal faces in $P_G$ that are not simplices.

For the special case of a tree $T$, we present a recursive way to compute the $f$-vector of the cosmological polytope $P_T$ via the $f$-vectors of smaller trees in Section~\ref{sec:trees}.
Such a recursive relation is based on the geometric realization of the cosmological polytope $P_{G'}$ as a pyramid over a bipyramid over a cosmological polytope $P_G$ if $G'$ obtained from $G$ by adding a leaf.
As a byproduct this yields that the normalized volume of the cosmological polytope of any tree with $e$ edges is $4^e$.
This geometric construction also lies behind the recursive formulae for the wavefunction of the universe obtained in \cite{arkani2017cosmological} via the frequency representation of the propagators.
For example, for a path graph $\Pi_n$ on $n$ nodes, we obtain the following recursion for the $f$-polynomial $f_{\Pi_n}(t)$ of the polytope $P_{\Pi_n}$
\[
    f_{\Pi_{n+2}}(t)=(1+t)((1+2t)f_{\Pi_{n+1}}(t) - t^2(1+t)f_{\Pi_{n}}(t)).
\]

We close in Section~\ref{sec:counting} by applying our methods to counting specific classes of faces of cosmological polytopes.
Theorem~\ref{thm:counting_low_dim} gives exact formulae for the number of $1$- and $2$-dimensional faces of cosmological polytopes.
Subsequently, we count simplex faces of cosmological polytopes of graphs with one cycle in Section~\ref{sec:counting_simplex}.
For example, the total number of simplex faces of the cycle graph on $n$ nodes is
$5^n-2^{n+1}$.

\subsection{Acknowledgments}
We would like to thank Bernd Sturmfels for introducing us to the topic of cosmological polytopes.
Furthermore, we are very grateful to Paolo Benincasa for numerous insightful discussions about these polytopes and their physical background.
Moreover, we thank Martina Juhnke and Justus Bruckamp for drawing our attention to a mistaken recitation of Theorem~\ref{thm:facets} in a previous version of this article.
Last but not least, we would like to thank the anonymous referees for their careful reading which indeed helped to improve the exposition.

The first author is supported by the Deutsche Forschungsgemeinschaft (DFG, German Research Foundation) -- SFB-TRR 358/1 2023 -- 491392403.

\section{Preliminaries}
In this section, we recall standard definitions and previous results on cosmological polytopes.
We refer to~\cite{Ziegler} for an in-depth introduction to polytopes and to \cite{arkani2017cosmological,Ben22} for a detailed introduction to cosmological polytopes.

A \emph{polytope} $P\subseteq \R^d$ is the convex hull of finitely many points in $\R^d$.
A \emph{face} $F\subseteq P$ is the set of points in the polytope $P$ that maximizes a linear functional $\phi:\R^d\to \R$.
The dimension of a face is the dimension of the affine space spanned by its points.
Faces of dimension $\dim(P)-1$ are called \emph{facets}.
Each polytope has finitely many faces and their numbers are counted in the \emph{$f$-vector} $f(P)$ of $P$ which is
$f(P)=(f_{-1},f_0,\dots,f_{\dim(P)})$ where $f_i$ is the number $i$-dimensional faces of $P$ and we set $f_{-1}=f_{\dim(P)}=1$.

Our analysis of the facial structure of cosmological polytopes relies on the following characterization of the facets of a cosmological polytope proved by Arkani-Hamed, Benincasa and Postnikov.

\begin{theorem}\label{thm:facets}[\cite{arkani2017cosmological}]
Facets of $P_G$ are in bijection with connected subgraphs $H=(V_H,E_H)$ of $G$. Under this bijection, a subgraph $H$ corresponds to the facet $F_H$ with all vertices of $P_G$ except
\begin{itemize}
    \item $p_e$ for an edge $e\in E_H$; 
    \item $p_{e,v}$ for an edge $e=\{v,w\}$ of $G$ with $v\in V_H$ and $e\notin E_H$.
\end{itemize}
So in particular, if an edge $e=\{v,w\}$ is not in $E_H$ but both $v$ and $w$ are in $V_H$ both vertices $p_{e,v}$ and $p_{e,w}$ are not part of the facet $F_H$.

Moreover, the facet $F_H$ is the intersection of $P_G$ with the following hyperplane
\[
    \sum_{v\in V_H} x_v + \sum_{\substack{e=\{v,w\},\\v\in V_H, w\notin V_H}}y_e+\sum_{\substack{e=\{v,w\}\notin E_H,\\v\in V_H, w\in V_H}}2y_e=0.
\]
\end{theorem}

The facet $F_G$ associated with the entire graph is called the \emph{scattering facet}.

\section{Face structure of cosmological polytopes}

\subsection{General faces}\label{sec:gen_faces}
We start by giving a criterion that characterizes the faces of cosmological polytopes.

\begin{theorem}\label{thm:faces}
Let $G=(V,E)$ be an undirected graph.
A set of vertices $X\subseteq V(P_G)$ defines a face of $P_G$ if and only if $X$ satisfies both of the following two conditions.
\begin{enumerate}
\item[(i)] If for a node $v\in V$ the set $X$ contains the vertices $p_e$ and $p_{e,v}$ for an edge $e=\{v,w\}\in E$ then $X$ contains the vertices $p_{e'}$ and $p_{e',v}$ for all edges $e'=\{v,w'\}\in E$.
\item[(ii)] If $X$ contains a subset $\{p_{e_1,v_1},p_{e_2,v_2}\dots, p_{e_k,v_k}\}$ for a cycle $e_1=\{v_1,v_2\},\ldots, e_k=\{v_k, v_1\}$ in $G$, then $X$ also contains the subset $\{p_{e_1,v_2},p_{e_2,v_3},\dots, p_{e_k,v_{1}}\}$.
\end{enumerate}
\end{theorem}

\begin{proof}

First, we show that every set of vertices $X$ satisfying the conditions (i) and (ii) defines a face of $P_G$.
We do this by proving that $X$ is an intersection of facets of $P_G$.
Let $y$ be a vertex of $P_G$ with $y\notin X$. It suffices to find a facet $F_H$ with $X\subseteq F_H$ and $y\notin F_H$. There are two cases: $y=p_e$ or $y=p_{e,v}$ for some $e\in E$ and $v\in V$. 

\begin{description}
\item[Case 1] Suppose $y=p_e$.
Let us define $H=(V_H,E_H)$ to be the connected component of the edge induced subgraph $\{ e\,|\, p_e\notin X\}$ containing the edge $e$.
Then the facet $F_H$ contains all vertices of $P_G$ except the vertices $p_{e'}$ with $e'\in E_H$ and the vertices $p_{e'',v''}$ with $e''=\{v'',w''\}\notin E_H$ and $v''\in V_H$. 
Since $e\in H$, the facet $F_H$ does not contain $y$.
Moreover by construction, $F_H$ contains all vertices in $X$ of the form $p_{e'}$ for some $e'\in E$.

Secondly, consider a vertex $p_{e'',v''}\in X$ for some $e''=\{v'',w''\}$.
Assume for a contradiction that $p_{e'',v''}\not\in F_H$.
Thus by Theorem~\ref{thm:facets} this means that $v''\in V_H$, but $e''\not\in E_H$ and thus $p_{e''}\in X$.
Since $H$ is connected, there exists an edge $f\in E_H$ that is adjacent to $v''$.
By condition (i), the vertex $p_{f}\in X$ which contradicts the construction of $H$.

\item[Case 2] Let $y = p_{e,v}$ with $e=\{v,w\}$.
We define a connected subgraph $H$ inductively. Let $H_{0} = v$  be the vertex $v$ itself. The subgraph $H_{i+1}$ is defined from $H_{i}$ in the following way:
\[
H_{i+1} := H_i \cup \{e'=\{v',w'\} \,|\, v'\in H_i, e'\notin H_i, \text{ and } p_{e',v'} \in X \}.
\]
Since $H_i\subseteq H_{i+1}$ and $G$ is a finite graph, the sequence $(H_i)_{i\in \mathbb{N}}$ stabilizes. We define $H$ to be the limit of the sequence $(H_i)_{i\in \mathbb{N}}$, by construction $H$ is connected. 
We need to show the following three things:
\begin{enumerate}
    \item Let $X_1:=\{p_{e'}\mid p_{e'}\in X\}$. We need to show $X_1\subseteq F_H$.
    We inductively show $X_1\subseteq F_{H_i}$ for all $i\ge 0$ which implies the claim.
    The case of $F_{H_0}$ is trivial as this facet contains all vertices of the form $p_{e'}$ in $P_G$.
    Next we show $X_1\subseteq F_{H_1}$.
    Consider the edge $e'=\{v,w'\}\in H_1$ and assume that $p_{e'}\in X$.
    By construction of $H_1$, we have $p_{e',v}\in X$ hence by condition (i), the vertex $y=p_{e,v}$ must also be in $X$ which contradicts the assumption that $y\notin X$. Hence, $p_{e'}\not\in X$ and $X_1\subseteq F_{H_1}$.
    
    So now assume $X_1\subseteq F_{H_i}$ for some $i\ge 1$.
    Consider again an edge $e'=\{v',w'\}\in H_{i+1}\setminus H_i$ with $v'\in H_i$.
    Then there exists an edge $e''=\{v',w''\}\in H_i$ since $H_i$ is connected and $i\ge 1$.
    By induction, this implies that $p_{e''}\not \in X$ and by construction of $H_{i+1}$ the vertex $p_{e',v'}\in X$.
    Hence, $p_{e'}\not \in X$ by condition (i) and thus $X_1\subseteq F_{H_{i+1}}$.
    \item Let $X_2:=X\setminus X_1$. Secondly, we show that $X_2\subseteq F_H$.
    This follows from the construction of $H$ as if there is a vertex $p_{e',v'}\in X\setminus F_H$ we would add the edge $e'$ to $H$ which contradicts the definition of~$H$.
    \item Lastly, we need to show $y\notin F_H$.
    Assume that $y=p_{e,v}\in F_H$.
    This means that $e\in H$.
    Therefore there exists a cycle in $H$ of the form $e_1=\{v_1,v_2\}$, $e_2=\{v_2,v_3\}$, \dots, $e_r=\{v_r,v_1\}$ with $e=e_r$, $v=v_1$, and $e_i\in H_i\setminus H_{i-1}$.
    Thus, $p_{e_i,v_{i}}\in X$ for all $i=1,\dots,r$ by construction of $H_i$.
    By condition~(ii) this implies that $p_{e_i,v_{i+1}}\in X$ for all $i\in \Z/r\Z$.
    In particular $y=p_{e_r,v_1}\in X$ which contradicts the assumption on $y$.
    Therefore, $y\notin F_H$.
\end{enumerate}
\end{description}

For the converse, we need to show that every face of $P_G$ satisfies the condition (i) and (ii).
First, notice that if two subsets $X$ and $Y$ satisfy both conditions so does their intersection $X\cap Y$.
Therefore it suffices to show conditions (i) and (ii) for the facets of $P_G$.

Let $H=(V_H,E_H)$ be a connected subgraph of $G$ and $F_H$ its associated facet.
Consider a node $v\in V$ and assume that  $p_e$ and $p_{e,v}$ are in $F_H$, hence $E_H$ does not contain the edge $e$ as well as any other edge adjacent to the node $v$. Hence by definition of $F_H$, the vertices $p_{e'},p_{e',v}$ are in $F_H$ for all edges $e'$ adjacent to the node $v$.
So $F_H$ satisfies the condition (i).

For the condition (ii), let $F_H$ contain a subset $\{p_{e_1,v_1},p_{e_2,v_2},\dots, p_{e_k,v_k}\}$ for a cycle $e_1=\{v_1,v_2\},\ldots, e_k=\{v_k, v_1\}$ in $G$.  First, assume that $\{e_1,\ldots,e_k\}\subseteq E_H$. In this case, by construction of $F_H$ it contains all the vertices $\{p_{e_1,v_2},p_{e_2,v_3},\dots, p_{e_k,v_{1}}\}$ as well, so the condition (ii) is satisfied.

In the case when $\{e_1,\ldots,e_k\}\not\subseteq E_H$, any edge adjacent to the nodes $v_1,\ldots,v_k$ can not be contained in $H$ as otherwise one of the vertices $p_{e_i,v_i}$ would be excluded from $F_H$. But then, by construction of $F_H$ vertices $\{p_{e_1,v_2},p_{e_2,v_3},\dots, p_{e_k,v_{1}}\}$ are  contained in the facet $F_H$, so the condition (ii) is satisfied.
\end{proof}
An immediate corollary of this theorem yields a complete description of the edge graph $\Gamma_G$ of $P_G$.

\begin{corollary}\label{cor:edge_graph}
The edge graph $\Gamma_G$ of $P_G$ is a complete graph on the vertices of $P_G$ with the following edges removed:
\begin{enumerate}
    \item $\{p_e,p_{e,v}\}$ for any edge $e$ and a non-leaf node $v$ of $G$.
    \item  $\{p_{e,v_1},p_{e',v_2}\}$ for a pair of parallel edges $e, e'$ between the nodes $v_1$ and $v_2$.
\end{enumerate}
\end{corollary}

A more general statement provides a description of all simplex faces of $P_G$.

\begin{theorem}\label{thm:simplex}
Let $G=(V,E)$ be an undirected graph.
A set of vertices $X\subseteq V(P_G)$ defines a simplex face of~$P_G$ if and only if 
\begin{enumerate}
\item[(i)] the induced subgraph of the edge graph $\Gamma_G$ to the vertex set $X$ is a complete graph, and
\item[(ii)] $X$ does not contain a subset $\{p_{e_1,v_1},\dots p_{e_k,v_k}\}$ for a cycle $e_1=\{v_1,v_2\},\ldots, e_k=\{v_k, v_1\}$ in $G$.
\end{enumerate}
\end{theorem}
\begin{proof}
By Corollary~\ref{cor:edge_graph} if the subgraph of $\Gamma_G$ induced by the vertices $X$ is a complete graph, $X$ must satisfy property (i) in Theorem~\ref{thm:faces}. Hence a subset $X$ satisfying the properties (i) and (ii) is a face of $P_G$ by Theorem~\ref{thm:faces}.
Moreover, the same properties (i) and (ii) are satisfied for every subset of vertices of the set $X$.
Hence, any subset of vertices of $X$ defines a face of $P_G$, so the vertices of $X$ form a simplex face.

For the converse, if $X$ defines a simplex face, then the induced subgraph of $\Gamma_G$ to $X$ is a complete graph and $X$ satisfies (ii) by Theorem~\ref{thm:faces}.
\end{proof}

\begin{remark}
The face structure of cosmological polytopes was also studied in \cite{ BT22,benincasa2020steinmann} from a slightly different perspective. More concretely, these works study which collections of facets of $P_G$ intersect in a face of expected codimension. The main tool in this analysis is the connection to the residues of the canonical form~${\mathbf \Omega}_{P_G}$.
\end{remark}

\subsection{Special faces}\label{sec:special_faces}
In this subsection, we will discuss two types of special faces appearing in cosmological polytopes which are the minimal non-simplex faces of cosmological polytopes (see Corollary~\ref{cor:minnonsim}).
We call the ones described in Proposition~\ref{prop:vertexface} \emph{vertex faces} and the ones described in Proposition~\ref{prop:cycleface} \emph{cycle faces}.

\begin{proposition}\label{prop:vertexface}
Let $v\in V$ be a vertex of $G$ of degree $d$ with the adjacent edges $e_1=\{v,w_1\},\ldots,e_d=\{v, w_d\}$. Then the cosmological polytope $P_G$ has a face of dimension $d$ with the $2d$ vertices given by
\[
p_{e_1}, p_{e_1, v},\ldots, p_{e_d}, p_{e_d, v}.
\]
We call this face a \emph{vertex face} $F_v$.

Moreover, the face $F_v$ is combinatorially equivalent to a $d$-dimensional cross-polytope.
\end{proposition}
\begin{proof}
The first statement follows directly from Theorem~\ref{thm:faces} as the set of vertices $\{p_{e_1}, p_{e_1, v},\ldots, p_{e_d}, p_{e_d, v}\}$ clearly satisfies conditions (i) and (ii). 

For the proof of the second statement notice that a shift of the face $F_v$ by the vector $-\mathbf{x}_v$ is the convex hull of the points
\[
\mathbf{y}_{e_1}-\mathbf{x}_{w_1}, \mathbf{x}_{w_1}-\mathbf{y}_{e_1}, \ldots , \mathbf{y}_{e_d}-\mathbf{x}_{w_d}, \mathbf{x}_{w_d}-\mathbf{y}_{e_d}. 
\]
Since the vectors $\mathbf{y}_{e_1}-\mathbf{x}_{w_1},\ldots, \mathbf{y}_{e_d}-\mathbf{x}_{w_d}$ are linearly independent, the vertex face $F_v$ is affinely (and in particularly combinatorially) equivalent to a $d$-dimensional cross-polytope.
\end{proof}

Recall that the $d$-dimensional \emph{cyclic polytope} $C(n,d)$ with $n$ vertices is the convex hull of $x(t_1),\dots,x(t_n)$ where $t_1<t_2<\dots<t_n$ are real numbers and $x:\R\to\R^d,t\mapsto(t,t^2,t^3,\dots,t^d)$ is a parametrization of the moment curve. It is known that cyclic polytopes are simplicial, i.e.\ all its proper faces are simplices (see for example \cite{gale1963neighborly}).
By Gale's evenness condition (\cite[Theorem 0.7]{Ziegler}), for a set of indices $I\subset [n]$ of size $d$, the corresponding set of vertices $\{x(t_i)\}_{i\in I}$ form a facet of $C(n,d)$ if and only if any two elements in $[n]\setminus I$ are separated by an even number of elements from $[n]$.

\begin{proposition}\label{prop:cycleface}
Let $e_1=\{v_1,v_2\},\ldots, e_d=\{v_d, v_1\}$ be a cycle $\sigma$ of length $d$ in $G$ with $v_i\neq v_j$ for $1\le i < j \le d$ and $d>1$.
Then the cosmological polytope $P_G$ has a face of dimension $2d-2$ with the $2d$ vertices:
\[
p_{e_1,v_1}, p_{e_1,v_2},\ldots,p_{e_d,v_d}, p_{e_d,v_1}.
\]
We call this face a \emph{cycle face} $F_\sigma$.

Moreover, the face $F_\sigma$ is combinatorially equivalent to a cyclic polytope of dimension  $2d-2$ with $2d$ vertices.
\end{proposition}
\begin{proof}
The first statement follows directly from Theorem~\ref{thm:faces} as the set of vertices 
\[
X=\{p_{e_1,v_1}, p_{e_1,v_2},\ldots,p_{e_d,v_d}, p_{e_d,v_1}\}
\]
clearly satisfies conditions (i) and (ii) of this theorem.

For the second statement, we will show that analogously to $C(2d,2d-2)$, facets of $F_\sigma$ are described by Gale's evenness condition which implies that the polytopes are combinatorially equivalent. 

First notice that if a subset $Y\subseteq X$ defines facet of $F_\sigma$, then $|Y|\geq 2d-2$. Moreover, by condition (ii) of Theorem~\ref{thm:faces}, to define a proper face of $F_\sigma$, the set $Y$ should not contain at least one of the points of type $p_{e_i,v_i}$ and $p_{e_j,v_{j+1}}$ for $i,j\in \Z/d\Z$. Therefore, $|Y|=2d-2$, i.e. every facet of $F_\sigma$ is a simplex. Finally, notice that the condition that $X\setminus Y = \{p_{e_i,v_i},p_{e_j,v_{j+1}}\}$ for some $i,j\in \Z/d\Z$ is equivalent to Gale's evenness condition if we order the elements of $X$ in the following way:
\[
(p_{e_1,v_1}, p_{e_1,v_2},\ldots,p_{e_d,v_d}, p_{e_d,v_1}). \qedhere\]
\end{proof}

\begin{remark}
Note that the cycle face $F_\sigma$ of $P_G$ corresponding to a cycle $\sigma$ of $G$ coincides with the the scattering facet of $P_{\sigma}$.
More generally, for any subgraph $H\subset G$, the scattering facet of $P_H$ appears as a face of $P_G$.
\end{remark}

We can now characterize the minimal non-simplex faces of a cosmological polytope, i.e. the faces that are combinatorially a simplicial polytope but not a simplex.
\begin{corollary}\label{cor:minnonsim}
A minimal non-simplex face $F$ of a cosmological polytope is either a vertex face or a cycle face.
\end{corollary}
\begin{proof}
The above propositions imply that the vertex faces are cross-polytopes and the cycle faces are cyclic polytopes which are both simplicial polytopes but not simplices.

So assume that $F$ is a simplicial polytope that does not contain vertices of a vertex face or a cycle face.
Condition (i) in Theorem~\ref{thm:faces} together with the assumption that $F$ does not contain a vertex face now implies that for every non-leaf edge $e=\{v,w\}$ the face $F$ can contain at most one of the vertices $p_e$, $p_{e,v}$, and $p_{e,w}$.
The assumption that $F$ does not contain a cycle face implies that for any pair of parallel edges $e,e'$ between the nodes $v_1,v_2$ the face $F$ can only contain at most one of the vertices $p_{e,v_1}$ and $p_{e',v_2}$.
By Corollary~\ref{cor:edge_graph} this implies that the edge graph of $F$ is a complete graph on its vertices.
Moreover, the same assertion is true for every subset of vertices of the face $F$ which implies that $F$ is combinatorially a simplex.
\end{proof}

\section{Trees}\label{sec:trees}
In this section we investigate the cosmological polytopes associated to the trees. Our main tool is the following proposition which describes how the cosmological polytope $P_G$ changes after adding a leaf to the graph $G$.

\begin{proposition}\label{prop:(bi)pyramid}
Let $G$ be a graph and let $G'$ be the graph that arises from $G$ by adding an vertex $v$ and an edge $e=\{v,w\}$ for some vertex $w$ of $G$. Then the following holds:
\begin{itemize}
    \item[(i)] The cosmological polytope $P_{G'}$ has a facet $F$ containing all vertices except of $p_{e,v}$. In particular, $P_{G'}$ is a pyramid over $F$ with apex $p_{e,v}$.
    \item[(ii)] The facet $F$ is a bipyramid over $P_G$ with apices $p_e$ and $p_{e,w}$ with the interval between $p_e, p_{e,w}$ intersecting $P_G$ in the interior of the vertex face $F_w$ defined in Proposition~\ref{prop:vertexface}.
\end{itemize}
\end{proposition}
\begin{proof}\phantom{a}
\begin{enumerate}
    \item[(i):] By Theorem~\ref{thm:facets} the facet $F$ corresponding to the subgraph $\{v\}$ in $G'$ contains all vertices of $P_{G'}$ except~$p_{e,v}$.
    Thus, $P_{G'}$ is a pyramid over $F$ with apex $p_{e,v}$.
    \item[(ii):] Filtered by the $x_e$ coordinate, the vertices of $F$ come in three layers: The layer $x_e=-1$ contains the vertex $p_{e}$, the layer $x_e=0$ the vertices in $P_G$ and the layer $x_e=1$ the vertex $p_{e,w}$.
    The interval between $p_e$ and $p_{e,w}$ intersects the $x_e=0$ layer in the point $\mathbf{x}_w$ which is in the interior of the face $F_w$.
    This implies the claim.\qedhere
\end{enumerate}
\end{proof}

One corollary of Proposition~\ref{prop:(bi)pyramid} is the computation of the volume of the cosmological polytopes of trees.

\begin{corollary}
Let $G, G'$ be as before, then one has $\Vol(P_{G'})=4\Vol(P_G)$, where $\Vol$ is the  normalized volume. In particular, for any tree $T$ with $e$ edges, the normalized volume of the cosmological polytope $P_T$ equals $4^{e}$.
\end{corollary}
\begin{proof}
One can check that the facet $F$ is a union of two pyramids of lattice height $1$  over $P_G$ which shows that $\Vol(F)=2\Vol(P_G)$. Moreover, the cosmological polytope $P_{G'}$ is a pyramid of lattice height $2$ over $F$, so the normalized volume of $P_{G'}$ is computed as 
\[
\Vol(P_{G'})=2\Vol(F)=4\Vol(P_G).\qedhere
\]
\end{proof}

\begin{proposition}\label{prop:upperf}
For a cosmological polytope $P_G$ and a vertex $w\in V(G)$ it holds that $f_{P_{G\setminus w}}(t)$ equals the ``upper $f$-vector'' of the vertex face $F_w$ in $P_G$, that is the $f$-vector of faces containing $F_w$.
\end{proposition}
\begin{proof}
Consider the partition of the vertices of $P_G$ by coordinate $\mathbf{x}_w$:
\begin{description}
    \item[Claim 1] There is a bijection $\phi$ of facets containing $F_w$ in $P_G$ and facets of $P_{G\setminus w}$.
Specifically, the facets of $P_G$ which contain $F_w$ are determined by connected subgraphs $H$ of $G$ which do not contain $w$. Such subgraphs are in bijection with the connected subgraphs of $G\setminus w$ which in turn determines facets of $P_{G\setminus w}$. The map $\phi$ maps a facet of $P_G$ given by a connected subgraph of $G$ that avoids $w$ to the facet of $P_{G\setminus w}$ given by the corresponding connected subgraph of $G\setminus w$.
\item[Claim 2] This bijection extends to a bijection between the faces containing $F_w$ in $P_G$ and all faces of $P_{G\setminus w}$.
Specifically, we consider the following map
\begin{align*}
    \phi: \{ \sigma \supseteq F_w\mid \sigma \mbox{ a face of }P_G \} & \to \{ \tau \mid \tau \mbox{ a face of }P_{G\setminus w} \}\\
    \sigma & \mapsto \tau = \bigcap_{\sigma \subset F, \,F \text{ is a facet of $P_G$}}\phi(F). 
\end{align*}
This map respects the dimension of the faces, that is $\dim (\phi(\sigma))=\dim(\sigma)-\dim(F_w)$.
\item[Claim 3]
This implies that the upper $f$-vector of $F_w$ in $P_G$ equals the (shifted) $f$-vector of $P_{G\setminus w}$.\qedhere
\end{description}
\end{proof}

\begin{definition}
For a polytope $P$, with $f$-vector $(f_{-1},\ldots,f_{\dim P})$ we will define its $f$-polynomial $f_P(t)$ to be 
\[
f_P(t) = \sum_{i=-1}^{\dim P} f_{i}t^{i+1}
\]
\end{definition}

We can use Proposition~\ref{prop:(bi)pyramid} to get a recursive relation for the $f$-polynomial of cosmological polytopes of the graphs $G$ and $G'$.
\begin{theorem}
The $f$-polynomials of $P_G$ and $P_{G'}$ are related in the following way:
\begin{align*}
f_F(t) &= (1+2t)f_{P_G}(t) - t^{\operatorname{deg}(w)+1}(1+t)f_{P_{G\setminus w}}(t)\\
f_{P_{G'}} &= (1 + t)f_F(t)
\end{align*}
\end{theorem}
\begin{proof}
Indeed, the $f$-polynomial of a pyramid $P$ with base $F$ is given by $f_P=(1+t)f_F(t)$ as every face of $F$ of dimension $d$ produces two faces of $P$ one of dimensions $d$ and another of dimension $d+1$. The $f$-polynomial of a the generic bipyramid over a polytope $P_G$ can be computes as $(1+2t)f_{P_G}(t)$. The description of faces of non-generic bipyramid follows, for example, from \cite[Proposition 2.3]{McMullen} as it is a particular example of subdirect sum: $F=(I,I)\oplus(F_v,P_G)$, where $I$ is an interval, and $F_v$ is a vertex face corresponding to the node $v$ of $G$. The face count involves the correction of $(1+2t)f_{P_G}(t)$ by the generating polynomial of the number of faces of $P_G$ containing $F_v$. Using Proposition~\ref{prop:upperf}, we obtain
\[
f_F(t) = (1+2t)f_{P_G}(t) - t^{\operatorname{deg}(w)+1}(1+t)f_{P_{G\setminus w}}(t),
\]
which finishes the proof.
\end{proof}

This gives an inductive way of computing the $f$-vector of cosmological polytopes of trees.
In the case of paths this yields the following recursion for their $f$-polynomials.

\begin{corollary}
Let $\Pi_n$ be the path graph on $n$ vertices.
Then we have the following recursion for the $f$-vector $f_{\Pi_n}(t)$ of the cosmological polytopes $P_{\Pi_n}$:
\[
    f_{\Pi_{n+2}}(t)=(1+t)((1+2t)f_{\Pi_{n+1}}(t) - t^2(1+t)f_{\Pi_{n}}(t)),
\]
and $f_{\Pi_1}(t)=t+1$, $f_{\Pi_2}(t)=t^3+3t^2+3t+1$.
The number of all faces of $P_{\Pi_n}$ is the evaluation of this recursion at $t=1$ which is the sequence \href{https://oeis.org/A154626}{A154626} in the Online Encyclopedia of Integer Sequences (OEIS).
\end{corollary}

\section{Counting faces}\label{sec:counting}
In this section we use our general description of faces of cosmological polytopes to compute their number in certain examples.
To simplify the exposition and obtain closed formulae for the face numbers, we assume that the graph $G$ does not have loops throughout this section. 

\subsection{Low-dimensional faces}
The main result of this subsection are formulas for the number of edges and $2$-dimensional faces of cosmological polytopes.
\begin{theorem}\label{thm:counting_low_dim}
Let $G=(V,E)$ be an undirected graph where $e$ is the number of edges, $l$ the number of leaves, $v_2$ the number of vertices of degree $2$, and $\Delta_i$ the number of cycles of length $i$ in $G$. The characterization of faces then yields:
\begin{enumerate}
    \item The number of edges of the cosmological polytope $P_G$ is
    \[
        f_1(P_G) = \binom{3e}{2} - 2e + l - 2\Delta_2,
    \]
    \item For a simple graph $G$, the number of $2$-dimensional faces of the cosmological polytope $P_G$ is:
    \[
        f_2(P_G) = 27\binom{e}{3} + 3(e+l)(e-1) + v_2 - 2\Delta_3,
     \]
\end{enumerate}
\end{theorem}
\begin{proof}
The first part follows directly from Corollary~\ref{cor:edge_graph}. Indeed, the number of edges of the complete graph on $3e$ vertices is $\binom{3e}{2}$; the number of removed edges of type $\{p_{e,v},p_e\}$ for each non-leaf node $v$ is $2e-l$ and the number of edges $p_{e,v_1},p_{e',v_2}$ for a pair of parallel edges $e, e'$ between the nodes $v_1$ and $v_2$ is $2\Delta_2$.

The second part is deduced similarly. By Corollary~\ref{cor:minnonsim} and the discussion thereafter, there are only two types of faces of dimension 2: triangles and quadrilaterals. First let us count the number of triangular faces of $P_G$. For this we first count the number of complete subgraphs of size 3 of the edge graph $\Gamma_G$ of $P_G$. Since $G$ is simple, from the description of Corollary~\ref{cor:edge_graph} it follows that the number of complete subgraphs of size 3 of $\Gamma_G$ is $27\binom{e}{3} + 3(e+l)(e-1)$. Now for every cycle of length 3 in $G$, there are exactly 2 complete subgraphs in $\Gamma_G$ which do not satisfy condition~(ii) of Theorem~\ref{thm:simplex}. So the final count of triangles in a simple graph is given by $27\binom{e}{3} + 3(e+l)(e-1)- 2\Delta_3$. On the other hand, each quadrilateral is either a vertex face of a node of degree $2$, or a cycle face for a cycle of length $2$. However, in a simple graph there is no cycle of length $2$, hence the total number of $2$-dimensional faces is given by 
\[
27\binom{e}{3} + 3(e+l)(e-1) - 2\Delta_3 + v_2. \qedhere
\]
\end{proof}

It is possible to deduce a formula for the number of $2$-dimensional faces of $P_G$ for a general graph $G$. For this one has to take into account contributions coming from the banana subgraphs (the graph on two vertices with multiple parallel edges). In particular, one has to compute the number of 2-dimensional faces of banana graphs which we do in Example~\ref{ex:banana}. It amounts to careful bookkeeping to deduce the general formula for the number of 2-faces of $P_G$ from the general description of simplex faces in Theorem~\ref{thm:simplex} and the count in Example~\ref{ex:banana}.

\begin{example}\label{ex:banana}
Let $B_k$ be the banana graph consisting of two vertices and $k$ parallel edges between them for $k\ge 1$.
The cosmological polytopes $P_{B_1}$ and $P_{B_2}$ have one and five $2$-dimensional faces, respectively.
In general, for $k\ge 3$ we have
\[
    f_2(P_{B_k})=15\binom{k}{3}+3\binom{k}{2}.
\]
One way to prove this formula is as follows.
Using for example \texttt{polymake}~\cite{polymake} one can compute that $P_{B_2}$ has two triangles and $P_{B_3}$ has 21 triangles.
Since every triangle can involve vertices corresponding to at most three edges of $B_k$ this immediately yields that $P_{B_k}$ has $15\binom{k}{3}+2\binom{k}{2}$ many triangles for $k\ge 2$.
By Theorem~\ref{thm:faces}, we obtain that every quadrilateral of $P_{B_k}$ for $k\ge 3$ is a cycle face stemming from a cycle of length two in $B_k$.
The general formula now follows from the fact that there are $\binom{k}{2}$ many such cycles in $B_k$.
\end{example}

\subsection{Simplex faces}\label{sec:counting_simplex}
In this subsection we count simplex faces in particular graphs.
For a polytope $P$ we denote by $f_{k}^\Delta(P)$ the number of $k$-dimensional simplex faces of $P$.

\begin{proposition}
Let $C_n$ be the cycle graph on $n$ nodes.
Then for $1\le k\le 2n$ the cosmological polytope $P_{C_n}$ has
\begin{equation}\label{eq:cycle}
    f_{k-1}^\Delta(P_{C_n})=-2\binom{n}{k-n}+\sum_{i=0}^{\lfloor \frac{k}{2}\rfloor}\binom{n}{i}\binom{n-i}{k-2i}3^{k-2i}.
\end{equation}
In total, $P_{C_n}$ has $5^n-2^{n+1}$ many simplex faces.
\end{proposition}
\begin{proof}
Say the cycle graph $C_n$ has nodes $v_1,\dots,v_n$ and edges $e_1,\dots,e_n$ where $e_i=\{v_i,v_{i+1}\}$ for $i\in \Z/n\Z$.
The edge graph $\Gamma_{P_{C_n}}$ of $P_{C_n}$ consists of $n$ triples of vertices $p_{e_i}$, $p_{e_i,v_i}$, and $p_{e_i,v_{i+1}}$ where within a cluster only the last two vertices are joined by an edge in $\Gamma_{P_{C_n}}$ and all pairs of vertices between different clusters are joined by an edge.
Therefore, $\binom{n}{i}\binom{n-i}{k-2i}3^{k-2i}$ is the number of complete subgraphs of $\Gamma_{P_{C_n}}$ on $k$ vertices with exactly $i$ edges within an edge cluster as described above for $0\le i\le \lfloor \frac{k}{2}\rfloor$.
Hence, the sum in Equation~\eqref{eq:cycle} is the number of all complete subgraphs of $\Gamma_{P_{C_n}}$ on $k$ vertices.

By Theorem~\ref{thm:simplex} we need to exclude the complete subgraphs of $\Gamma_{P_{C_n}}$ that contain either all vertices $p_{e_i,v_i}$ or all vertices $p_{e_i,v_{i+1}}$ for $i \in \Z/n\Z$.
For $n\le k\le 2n$ there are exactly $2\binom{n}{k-n}$ such complete subgraphs which is the term we subtract in Equation~\eqref{eq:cycle}.

For the last statement, note that when we want to count all complete subgraphs of $\Gamma_{P_{C_n}}$ we have five choices for each of the $n$ vertex clusters: no vertex, one of the vertices $p_{e_i}$, $p_{e_i,v_i}$, and $p_{e_i,v_{i+1}}$ or both of the vertices $p_{e_i,v_i}$ and $p_{e_i,v_{i+1}}$.
We just need to exclude the choice of no vertices at all, which yields $5^n-1$ complete subgraphs in total.
We claim that for the complete subgraphs we need to exclude by the cycle condition exactly $2^{n+1}-1$ choices.
Indeed, once a complete subgraphs contains all vertices $p_{e_i,v_i}$ for $i \in \Z/n\Z$ then there are two choices in every cluster: the complete subgraph can contain the vertex $p_{e_i,v_{i+1}}$ or not. After counting the analogous possibilities for the complete subgraphs containing all vertices $p_{e_i,v_{i+1}}$ for $i \in \Z/n\Z$, we obtain the term $2^{n+1}-1$ we need to subtract since there is one configuration that appears in both of these versions.
\end{proof}

A similar argument yields the following generalization.
\begin{proposition}
Let $G=(V,E)$ be a graph with exactly one cycle. Say this cycle is of length $d$ and assume $d>2$.
Let $e=|E|$ and $l$ be the number of leaves of $G$.
Then the number simplex faces of $P_G$ is
\[
    6^l\cdot 5^{e-l-d}\cdot(5^d-2^{d+1}+1)-1.
\]
\end{proposition}

\bibliographystyle{amsplain} 
	\bibliography{biblio.bib}
\end{document}